\documentclass[11pt]{article}
\usepackage{amsmath,amssymb,amscd,verbatim,amsthm}
\usepackage{latexsym, epsfig, graphpap, graphics, mathrsfs}
\usepackage[varg]{pxfonts}
\usepackage{eqnarray,amsmath}

\newtheorem{theorem}{Theorem}[section]

\newtheorem{corollary}{Corollary}[section]

\begin{document}
 \begin{center}{\bf{ Some new $I_\lambda$-lacunary statistical convergence  }} \end{center}

\vskip 0.3 cm

\begin{center} Adem Kili\c{c}man \\ Department of Mathematics and Institute for Mathematical Research
University Putra Malaysia, 43400 Serdang, Selangor, MALAYSIA\\ e-mail: akilic@upm.edu.my\end{center}

\vskip 0.3 cm

\begin{center} Stuti Borgohain $^*${\footnote {The second author was supported under the Post Doctoral Fellow by National Board of Higher
Mathematics, DAE (Govt. of India), having project No. NBHM/PDF.50/2011/64.}} \\ Department of Mathematics, Indian Institute of Technology, Bombay, Powai : 400076, Mumbai, Maharashtra; INDIA \\ e-mail : stutiborgohain$@$yahoo.com \end{center}

\begin{abstract} In this paper, we introduce some new  $I_\lambda$-lacunary statistically convergent sequence spaces
of order $\alpha$ defined by a Musielak-Orlicz function. We study some relations between $I_\lambda$-lacunary
statistically convergence with lacunary $I_\lambda$-summable sequences. Moreover we also study about the
$I_\lambda$-lacunary statistically convergence of sequences in topological groups and give some important inclusion
theorems. \\


\end{abstract}


 \section*{Introduction} Fast \cite{Fast} introduced the concept of statistical convergence as the generalization of convergence of real sequences.  Later on, it  became as an important tool in the summability theory by many mathematicians. Also statistical convergence has been discussed in the theory of Fourier analysis, ergodic theory, number theory, measure theory, trigonometric series, turnpike theory and Banach spaces and many more. For references one may see [\cite{Fridy}, \cite{Orhan}, \cite{Freedman}, \cite{Schoenberg}, \cite{Malafossa}] and many more.\\

Kostyrko et al. \cite{Kostyrko} further extended the notion of statistical convergence to $I$-statistical convergence
and studied some basic properties on it. A family of sets $I \subseteq 2^X$ (power sets of $X$) is said to be an ideal
if $I$ is additive i.e. $P,Q \in I \Rightarrow P \cup Q \in I$ and hereditary i.e. $P \in I, Q \subseteq P \Rightarrow
Q \in I$, where $X$ is any non empty set. \\

A lacunary sequence is defined as an increasing integer sequence $\theta = (j_r)$ such that $j_0=0$ and
$h_r=j_r-j_{r-1} \rightarrow \infty$ as $r \rightarrow \infty.$\\

\noindent Note: Throughout this paper, the intervals determined by $\theta$ will be denoted by $J_r=(j_{r-1}, j_r]$
and the ratio $\frac{j_r}{j_{r-1}}$ will be defined by $\phi_r$.

\section*{Preliminary Concepts} A sequence $(x_i)$ of real numbers is said to be statistically convergent to $Z$ if,
for arbitrary $\xi>0$, the set $K(\xi)=\{i \in \mathbb{N}: \vert x_i -Z \vert \geq \xi \}$ has natural density zero,
i.e., $$\displaystyle\lim_i \frac{1}{i} \displaystyle\sum_{j=1}^i \chi_{K(\xi)}(j)=0,$$ where $\chi_{K(\xi)}$ denotes
the characteristic function of $K(\xi)$.

A sequence $(x_i)$ of elements of $\mathbb{R}$ is said to be $I$-convergent to $Z\in \mathbb{R}$ if for each $\xi>0$,
$$\{i \in \mathbb{N}: \vert x_i -Z \vert \geq \xi \} \in I.$$ For any lacunary sequence $\theta= (j_r)$, the space
$N_\theta$ defined as, (Freedman et al.\cite{Freedman}) $$N_\theta=\left\{(x_j): \displaystyle\lim_{r \rightarrow
\infty} j_r^{-1} \displaystyle\sum_{j \in J_r} \vert x_j - Z \vert =0, \mbox{~for some~} Z \right\}.$$ By an Orlicz
function , we mean a function $M: [0,\infty )\rightarrow [0,\infty )$, which is continuous, non-decreasing and convex
with $M(0) = 0, M(x)>0$, for $x>0$ and $M(x)\rightarrow \infty$, as $x\rightarrow \infty$.\\

Musielak \cite{Musielak} defined the concept of Musielak-Orlicz function as $\mathscr{M}=(M_j)$. A sequence
$\mathscr{N}=(N_k)$ defined by $$N_k(v)=\sup \{ \vert v \vert u -M_j(u): u \geq 0 \}, k=1,2,..$$ which is called the
complementary function of a Musielak-Orlicz function $\mathscr{M}$. The Musielak-Orlicz sequence space $t_\mathscr{M}$
and its subspace $h_\mathscr{M}$ are defined as follows: \begin{eqnarray*} t_\mathscr{M}&=&\{ x \in w:
I_\mathscr{M}(cx) < \infty \mbox{~for some ~} c>0\},\\ h_\mathscr{M}&=&\{x \in w:I_\mathscr{M}(cx) < \infty, \forall c
>0\},\end{eqnarray*} where $I_\mathscr{M}$ is a convex modular defined by, $$I_\mathscr{M}(x) =
\displaystyle\sum_{k=1}^\infty M_j(x_k), x=(x_k) \in t_\mathscr{M}.$$ It is considered $t_\mathscr{M}$ equipped with
the Luxemberg norm $$\Vert x \Vert = \inf \left\{ k>0: I_\mathscr{M}\left(\frac{x}{k}\right) \leq 1 \right\}$$ or
equipped with the Orlicz norm $$\Vert x \Vert^0 = \inf \left\{ \frac{1}{k} (1+I_\mathscr{M}(kx)): k>0 \right\}.$$

Let $\lambda=(\lambda_i)$ be a non-decreasing sequence of positive integers. We denote $\Lambda$ as the set of all
non-decreasing sequence of positive integers.  We consider a sequence $\{x_j \}_{j \in \mathbb{N}}$ which is said to
be $I_\lambda$-lacunary statistically convergent of order $\alpha$ to $Z$, if for each $\gamma>0$ and $\xi>0$,

$$\left\{i \in \mathbb{N}: \frac{1}{\lambda_i^\alpha}\left\vert \left\{j \leq i:  \frac{1}{h_i} \displaystyle\sum_{j
\in I_i} M_j \left(\frac{\vert x_j-Z \vert}{\rho^{(j)}} \right) \geq \gamma \right\}\right\vert \geq \xi \right\} \in
I.$$

We denote the class of all $I_\lambda$-lacunary statistically convergent sequences of order $\alpha$ defined by
Musielak-Orlicz function as $S^\alpha_{I_\lambda}(\mathscr{M}, \theta)$. \\

Some particular cases:

\begin{enumerate}

\item If $M_j(x)=M(x)$, for all $j \in \mathbb{N}$, then $S_{I_\lambda}^\alpha(\mathscr{M}, \theta)$ is reduced to

    $S_{I_\lambda}^\alpha(M, \theta)$. Further, if $M_j(x)=x$, for all $j \in \mathbb{N}$, then
    $S_{I_\lambda}^\alpha(\mathscr{M}, \theta)$ will be changed as $S_{I_\lambda}^\alpha(\theta)$.

\item If $\lambda_i=i$, for all $i \in \mathbb{N}$, then $S_{I_\lambda}^\alpha(\mathscr{M}, \theta)$ will be
    reduced to $S_{I}^\alpha(\mathscr{M}, \theta)$.

\item If $\alpha=1$, then $\alpha$-density of any set reduces to the natural density of the set. So, the set
    $S_{I_\lambda}^\alpha(\mathscr{M}, \theta)$ reduces to $S_{I_\lambda}(\mathscr{M}, \theta)$ for $\alpha=1$.

\item If $\theta=(2^r)$ and $\alpha=1$, then the sequence $(x_j)$ is said to be $I_\lambda$-statistically
    convergent defined by Musielak-Orlicz function, i.e. $(x_j) \in S_{I_\lambda}(\mathscr{M})$.

\item If $M_j(x)=x, \theta=(2^r)$, $\lambda_j=j$, $\alpha=1$, then $I_\lambda$-lacunary statistically convergence

    of order $\alpha$ defined by Musielak-Orlicz function reduces to $I$-statistical convergence.
\end{enumerate}

In this article, we define the concept of $I_\lambda$- lacunary statistically convergent of order $\alpha$ defined by
Musielak-Orlicz function and investigate some results on these sequences. Moreover, we study $I_\lambda$-lacunary
statistically convergence of sequences in topological groups and give some important inclusion theorems.

\section*{Main Results}

\begin{theorem} Let $\lambda=(\lambda_i)$ and $\mu=(\mu_i)$ be two sequences in $\Lambda$ such that $\lambda_i \leq
\mu_i$ for all $i \in \mathbb{N}$ and $0<\alpha \leq \beta \leq 1$ for fixed reals $\alpha$ and $\beta$. If
$\lim\displaystyle\inf_{i \rightarrow \infty} \frac{\lambda_i^\alpha}{\mu_i^\beta} >0$, then $S_{I_\mu}^\beta
(\mathscr{M}, \theta) \subseteq S_{I_\lambda}^\alpha(\mathscr{M}, \theta)$. \end{theorem}

\begin{proof} Suppose that $\lambda_i \leq \mu_i$ for all $ i \in \mathbb{N}$ and $\lim\displaystyle\inf_{i
\rightarrow \infty} \frac{\lambda_i^\alpha}{\mu_i^\beta} >0$. Since $I_i \subset J_i$, where $J_i=[i-\mu_i+1, i]$, so
for $\varepsilon >0$, we can write,

$$\left\{ j \in J_i: \vert x_j -Z \vert \geq \varepsilon \} \supset \{ j \in I_i:\vert x_j -Z \vert \geq \varepsilon
\right\}$$

which implies $$\frac{1}{\mu_i^\beta} \left\vert \left\{ j \in J_i: \vert x_j - Z \vert \geq \varepsilon \right\}
\right\vert \geq \frac{\lambda_i^\alpha}{\mu_i^\beta}. \frac{1}{\lambda_i^\alpha} \left\vert \left\{ j \in I_i:\vert
x_j - Z \vert \geq \varepsilon \right\} \right\vert,$$ for all $ i \in \mathbb{N}$.\\

Assume that $\lim\displaystyle\inf_{i \rightarrow \infty} \frac{\lambda_i^\alpha}{\mu_i^\beta}=a$,  so from the
definition we get $\left\{i \in \mathbb{B}:  \frac{\lambda_i^\alpha}{\mu_i^\beta} <\frac{a}{2} \right\}$ is finite.
Now for $\delta>0$, \begin{eqnarray*} \left\{ i \in \mathbb{N}:\frac{1}{\lambda_i^\beta} \left\vert \left\{ j \in J_i:
\vert x_j -Z \vert \geq \varepsilon \right\} \right\vert \geq \delta \right\}
 &\subset&  \left\{ i \in \mathbb{N}: \frac{1}{\mu_i^\alpha} \left\vert \left\{ j \in I_i : \vert x_j -Z \vert \geq
 \varepsilon \right\} \right\vert \geq \frac{a}{2} \delta \right\}\\
& \cup& \left\{ i \in \mathbb{N} : \frac{\lambda_i^\alpha}{\mu_i^\beta} < \frac{a}{2} \right\}. \end{eqnarray*} Since
$I$ is admissible and $(x_j)$ is $I_\mu$-lacunary statistically convergent sequence of order $\beta$ defined by a
Musielak-Orlicz function, so by using the continuity of Musielak-Orlicz function, the set on the right hand side with
the lacunary sequence $\theta=(h_i)$ belongs to $I$. This completes the proof. \end{proof}

\begin{theorem} If $\displaystyle\lim_{i \rightarrow \infty} \frac{\mu_i}{\lambda_i^\beta}=1$, for
$\lambda=(\lambda_i)$ and $\mu=(\mu_i)$ two sequences of $\Lambda$ such that $\lambda_i \leq \mu_i, \forall i \in
\mathbb{N}$ and $0<\alpha\leq \beta\leq 1$ for fixed $\alpha, \beta$ reals,   then  $S_{I_\lambda}^\alpha(\mathscr{M},
\theta) \subseteq S_{I_\mu}^\beta(\mathscr{M}, \theta)$. \end{theorem}

\begin{proof} Let $(x_j)$ be $I_\lambda$-lacunary statistically convergent to $Z$ of order $\alpha$ defined by the
Musielak-Orlicz function $\mathscr{M}$. Also assume that $\displaystyle\lim_{i \rightarrow \infty}
\frac{\mu_i}{\lambda_i^\beta} =1$. So we can choose $m \in \mathbb{N}$ such that $\left\vert
\frac{\mu_i}{\lambda_i^\beta}-1 \right\vert < \frac{\delta}{2},\forall i\geq m.$\\

Since $I_i \subset J_i$, for $\varepsilon >0$, we may write, \begin{eqnarray*} \frac{1}{\mu_i^\beta} \left\vert
\left\{ j \in J_i: \vert x_j -Z \vert \geq \varepsilon \right\} \right\vert
 &=& \frac{1}{\mu_i^\beta} \left\vert \left\{ i-\mu_i+1 \leq j \leq i-\lambda_i : \vert x_j -Z \vert \geq \varepsilon
\right\}\right\vert \\
&+& \frac{1}{\mu_i^\beta} \left\vert \left\{ j \in I_i :\vert x_j -Z\vert \geq \varepsilon \right\}\right \vert\\
&\leq& \frac{\mu_i-\lambda_i}{\mu_i^\beta} + \frac{1}{\mu_i^\beta} \left\vert \left \{ j \in I_i : \vert  x_j -Z \vert
\geq \varepsilon \right\} \right\vert\\ &\leq& \frac{\mu_i-\lambda_i^\beta}{\lambda_i^\beta} + \frac{1}{\mu_i^\beta}
\left \vert \left \{ j \in I_i : \vert x_j -Z \vert \geq \varepsilon \right\} \right\vert\\ &\leq& \left(
\frac{\mu_i}{\lambda_i^\beta}-1 \right) + \frac{1}{\lambda_i^\alpha} \left\vert \left\{ j \in I_i: \vert x_j -Z \vert
\geq \varepsilon \right\} \right\vert\\ &=& \frac{\delta}{2} + \frac{1}{\lambda_i^\alpha} \left \vert \left\{ j \in
I_i : \vert x_j - Z \vert \geq \varepsilon \right\} \right\vert . \end{eqnarray*} Hence, \begin{eqnarray*} \left\{ i
\in \mathbb{N}:\frac{1}{\mu_i^\beta} \left\vert \left\{ j \leq i: \vert x_j -Z \vert \geq \varepsilon \right\} \right
\vert \geq \delta \right\} &\subset& \left\{ i \in \mathbb{N}: \frac{1}{\lambda_i^\alpha} \left\vert \left\{ j \in I_i
: \vert x_j -Z \vert \geq \varepsilon \right\} \right\vert \geq \frac{\delta}{2} \right\}\\ &\cup& \{ 1,2,3,..m \}.
\end{eqnarray*}

Since $(x_j)$ is $I_\lambda$-lacunary statistically convergent sequence of order $\alpha$ defined by a Musielak-Orlicz
function $\mathscr{M}$ and since $I$ is admissible, so by using the continuity of Musielak-Orlicz function, we can say
that the set on the right hand side with the lacunary sequence $\theta=(h_i)$ belongs to $I$ and this proves that,
$$S_{I_\lambda}^\alpha(\mathscr{M}, \theta)\subseteq S_{I_\mu}^\beta(\mathscr{M}, \theta).$$ \end{proof}

We define the lacunary $I_\lambda$-summable sequence of order $\alpha$ defined by Musielak-Orlicz function as:
$$w_{I_\lambda}^\alpha(\mathscr{M},\theta)= \left\{ i \in \mathbb{N}: \frac{1}{\lambda_i^\alpha} \left( j \leq i:
\frac{1}{h_i} \displaystyle\sum_{j \in I_i} M_j \left( \frac{ \vert x_j - Z \vert }{ \rho^{(j)}} \right) \geq
\varepsilon \right) \right\} \in I.$$ \begin{theorem} Given for $\lambda=(\lambda_i), \mu=(\mu_i) \in \Lambda$.
Suppose that $\lambda_i \leq \mu_i$ for all $i \in \mathbb{N}, 0 < \alpha \leq \beta \leq 1$. Then, \begin{enumerate}
\item If $\lim\displaystyle\inf_{i \rightarrow \infty} \frac{\lambda_i^\alpha}{\mu_i^\beta} >0$, then
    $w_\mu^\beta(\mathscr{M}, \theta) \subset w_\lambda^\alpha(\mathscr{M}, \theta)$.

\item If $\displaystyle\lim_{i \rightarrow \infty} \frac{\mu_i}{\lambda_i^\beta} = 1$, then $\ell_\infty \cap
    w_\lambda^\alpha(\mathscr{M},\theta) \subset w_\mu^\beta(\mathscr{M},\theta)$.
\end{enumerate} \end{theorem}

\begin{theorem} Let $\alpha$ and $\beta$ be fixed real numbers such that $0 <\alpha \leq \beta \leq 1, \lambda_i \leq
\mu_i$ for all $i \in \mathbb{N}$, where $\lambda, \mu \in \Lambda$. Then, if $\lim\displaystyle\inf_{i \rightarrow
\infty} \frac{\lambda_i^\alpha}{\mu_i^\beta} >0$, and if $(x_j)$ is lacunary $I_\mu$-summable of order $\beta$ defined
by a Musielak-Orlicz function $\mathscr{M}$, then it is $I_\lambda$-lacunary statistically convergent of order
$\alpha$ defined by the Musielak-Orlicz function $\mathscr{M}$. \end{theorem}

\begin{proof} For any $\varepsilon >0$, we have, \begin{eqnarray*} \displaystyle\sum_{j \in J_i} \vert x_j -Z \vert
&=& \displaystyle\sum_{j \in J_i, \vert x_j -Z \vert \geq \varepsilon} \vert x_j -Z \vert + \displaystyle\sum_{j \in
J_i, \vert x_j -Z \vert < \varepsilon} \vert x_j -Z \vert\\ &\geq& \displaystyle\sum_{j \in I_i, \vert x_j - Z \vert
\geq \varepsilon} \vert x_j - Z \vert + \displaystyle\sum_{j \in I_i, \vert x_j -Z \vert \geq \varepsilon} \vert x_j
-Z \vert \\ &\geq& \displaystyle\sum_{j \in I_i, \vert x_j -Z \vert \geq \varepsilon} \vert x_j -Z \vert\\ &\geq&
\vert \{ j \in I_i: \vert x_j -Z \vert \geq \varepsilon \} \vert . \varepsilon \end{eqnarray*} so, \begin{eqnarray*}
\frac{1}{\mu_i^\beta} \displaystyle\sum_{j \in J_i} \vert x_j -Z \vert &\geq& \frac{1}{\mu_i^\beta} \vert \{ j \in I_i
: \vert x_j - Z \vert \geq \varepsilon \} \vert. \varepsilon\\ &\geq&\frac{\lambda_i^\alpha}{\mu_i^\beta}.
\frac{1}{\lambda_i^\alpha} \vert \{ j \in I_i: \vert x_j - Z \vert \geq \varepsilon \} \vert. \varepsilon .
\end{eqnarray*}

If $\lim\displaystyle\inf_{i \rightarrow \infty} \frac{\lambda_i^\alpha}{\mu_i^\beta} =a$ then, $\left\{ i \in
\mathbb{N}: \frac{\lambda_i^\alpha}{\mu_i^\beta} < \frac{a}{2} \right\}$ is finite. So, for $\delta >0$, we get,
\begin{eqnarray*} \left\{ i \in \mathbb{N}: \frac{1}{\lambda_i^\alpha}\left\vert \left \{ j \leq i:
\displaystyle\sum_{j \in J_i} \vert x_j - Z \vert \geq \varepsilon \right\} \right\vert \geq \delta \right\} &\subset&
\left\{ i \in \mathbb{N}: \frac{1}{\mu_i^\beta} \left \{ j \in I_i : \vert x_j - Z \vert \geq \varepsilon \right\}
\geq \frac{a}{2} \delta \right\} \\ &\cup& \left\{ i \in \mathbb{N} : \frac{\lambda_i^\alpha}{\mu_i^\beta} <
\frac{a}{2} \right\}. \end{eqnarray*}

Since $I$ is admissible and $(x_j)$ is lacunary $I_\mu$-summable sequence of order $\beta$ defined by a
Musielak-Orlicz function $\mathscr{M}$,  using the continuity of the Musielak-Orlicz function $\mathscr{M}$ and using
the lacunary sequence $\theta=(h_i)$, we can conclude that, $w_{I_\mu}^\beta(\mathscr{M}, \theta) \subseteq
S_{I_\lambda}^\alpha(\mathscr{M}, \theta)$. \end{proof}

\begin{theorem} Let $\displaystyle\lim_{i \rightarrow \infty} \frac{\mu_i}{\lambda_i^\beta}=1$, for fixed real numbers
$\alpha$ and $\beta$ such that $0< \alpha \leq \beta \leq 1$ and $\lambda_i \leq \mu_i$, for all $i \in \mathbb{N}$,
where $\lambda, \mu \in \Lambda$. Also let $\theta!$ be a refinement of $\theta$.  If a bounded sequence $(x_j)$ is
$I_\lambda$-lacunary statistically convergent sequence of order $\alpha$ defined by a Musielak-Orlicz function
$\mathscr{M}$, then it is also lacunary $I_\mu$-summable sequence of order $\beta$ defined by the Musielak-Orlicz
function $\mathscr{M}$. i.e. $ S_{I_\lambda}^\alpha(\mathscr{M},\theta) \subseteq w_{I_\mu}^\beta(\mathscr{M},
\theta!).$ \end{theorem}

\begin{proof} Suppose that $(x_j)$ is $I_\lambda$-lacunary statistically convergent sequence of order $\alpha$ defined
by a Musielak-Orlicz function $\mathscr{M}$.\\

Given that $\displaystyle\lim_{i \rightarrow \infty} \frac{\mu_i}{\lambda_i^\beta}=1$, so we can choose $m \in
\mathbb{N}$ such that $\left\vert \frac{\mu_i}{\lambda_i^\beta}-1 \right\vert < \frac{\delta}{2}, \forall i \geq m$.
\\

Assume that there is a finite number of points $\theta!=(j_i^!)$ in the interval $I_i=(j_{i-1}, j_i]$. Let there
exists exactly one point $j_i^!$ of $\theta!$ in the interval $I_i$, that is, $j_{i-1}=j _{p-1}^! < j_p^! <
j_{p+1}^!=j_i$. Let $$I_i^1=(j_{i-1},j_p], I_i^2=(j_p, j_i], h_i^1=j_p-j_{i-1}, h_i^2=j_i-j_p.$$ \begin{eqnarray*}
\frac{1}{\mu_i^\beta} \left(h_i^{-1} \displaystyle\sum_{j \in J_i} \vert x_j -Z \vert \right) &\leq&
\frac{1}{\mu_i^\beta} \left((h_i^{-1}h_i^1) h_i^{-1} \displaystyle\sum_{j \in I_i^1} \vert x_j -Z \vert
+(h_i^{-1}h_i^2) h_i^{-2} \displaystyle\sum_{j \in I_i^2} \vert x_j -Z \vert \right) \\ &\leq& \left(\frac{\mu_i -
\lambda_i}{\mu_i^\beta}\right) (h_i^{-1}h_i^1) h_i^{-1}M + \frac{1}{\mu_i^\beta} \left((h_i^{-1}h_i^2) h_i^{-2}
\displaystyle\sum_{j \in I_i^2} \vert x_j -Z \vert \right) \\ &\leq& \left(\frac{\mu_i -
\lambda_i^\beta}{\lambda_i^\beta}\right)(h_i^{-1}h_i^1) h_i^{-1} M+ \frac{1}{\mu_i^\beta} \left((h_i^{-1}h_i^2)
h_i^{-2} \displaystyle\sum_{j \in I_i^2} \vert x_j -Z \vert \right) \\ &\leq& \left(\frac{\mu_i
}{\lambda_i^\beta}-1\right)(h_i^{-1}h_i^1) h_i^{-1} M+ \frac{1}{\mu_i^\beta} \left((h_i^{-1}h_i^2) h_i^{-2}
\displaystyle\sum_{j \in I_i^2, \vert x_j - Z \vert \geq \varepsilon} \vert x_j -Z \vert \right)\\
&+&\frac{1}{\mu_i^\beta} \left((h_i^{-1}h_i^2) h_i^{-2} \displaystyle\sum_{j \in I_i^2, \vert x_j - Z \vert <
\varepsilon} \vert x_j -Z \vert \right) \\ &\leq& \left(\frac{\mu_i}{\lambda_i^\beta}-1 \right)(h_i^{-1}h_i^1)
h_i^{-1} M+ \frac{M}{\lambda_i^\alpha} \left\vert \left\{ j \in I_i: (h_i^{-1}h_i^2)h_i^{-2} \vert x_j-Z \vert \geq
\varepsilon \right\} \right\vert + \varepsilon (h_i^{-1}h_i^2) h_i^{-2}, \forall i \in \mathbb{N} \\
&=&\frac{\delta}{2}(h_i^{-1}h_i^1) h_i^{-1} M+ \frac{M}{\lambda_i^\alpha} \left\vert \left\{ j \in I_i:
(h_i^{-1}h_i^2)h_i^{-2} \vert x_j-Z \vert \geq \varepsilon \right\} \right\vert + \varepsilon (h_i^{-1}h_i^2)
h_i^{-2}. \end{eqnarray*}

Since $x \in w_{I_\mu}^\beta(\mathscr{M}, \theta!)$, so $0<h_i^{-1}h_i^1\leq 1$ and $0< h_i^{-1}h_i^2 \leq 1$.\\

Hence, for $\delta>0$,

\begin{eqnarray*} \left\{ i \in \mathbb{N}: \frac{1}{\mu_i^\beta} \left(\frac{1}{h_i} \displaystyle\sum_{j \in J_i}
\vert x_j-Z \vert \geq \varepsilon \right) \geq \delta \right\}&\subset& \left\{ i \in \mathbb{N}:
\frac{M}{\lambda_i^\alpha}\left\vert \left\{ j \in I_i: \frac{1}{h_i^2} \vert x_j-Z \vert \geq \varepsilon \right\}
\right\vert \geq \delta \right\}\\ &\cup& \{1,2,3,..m\}. \end{eqnarray*}

Since $(x_j)$ is $I_\lambda$-lacunary statistical convergent sequence of order $\alpha$ defined by the Musielak-Orlicz
function $\mathscr{M}$ and since $I$ is admissible, so by using the continuity of $\mathscr{M}$,we can say that the
set on the right hand side belongs to $I$ and this proves that, $$ S_{I_\lambda}^\alpha(\mathscr{M},\theta) \subseteq
w_{I_\mu}^\beta(\mathscr{M}, \theta!).$$ \end{proof}

\begin{corollary} Let $\lambda_ \leq \mu_i$ for all $i \in \mathbb{N}$ and $0< \alpha \leq \beta \leq 1$. Let
$\lim\displaystyle\inf_{i \rightarrow \infty} \frac{\lambda_i^\alpha}{\mu_i^\beta} >0$, $\theta!$ be the refinement of
$\theta$. Also let $\mathscr{M}=(M_i)$ be a Musielak-Orlicz function where $(M_i)$ is pointwise convergent. Then
$w_{I_\mu}^\beta(\mathscr{M}, \theta!)\subset S_{I_\lambda}^\alpha(\mathscr{M}, \theta)$ iff $\displaystyle\lim_i M_i
\left(\frac{\nu}{\rho^{(i)}}\right)>0$, for some $\nu>0, \rho^{(i)}>0$. \end{corollary}

\begin{corollary} Let $\mathscr{M}=(M_i)$ be a Musielak-Orlicz function and $\displaystyle\lim_{i \rightarrow \infty}
\frac{\mu_i}{\lambda_i^\beta} =1$, for fixed numbers $\alpha$ and $\beta$ such that $0< \alpha \leq \beta \leq 1$.
Then $S_{I_\lambda}^\alpha(\mathscr{M}, \theta) \subset w_{I_\mu}^\beta(\mathscr{M}, \theta)$ iff
$\displaystyle\sup_\nu \displaystyle\sup_i \left(\frac{\nu}{\rho^{(i)}}\right)$. \end{corollary}

\section*{$I_\lambda$-lacunary statistical convergence in topological group}

We define $I_\lambda$-lacunary statistical convergent of order $\alpha$ in topological groups. Let $\Lambda$  be the
collection of all non-decreasing sequence of positive integers $\lambda=(\lambda_i)$. We consider a  sequence $\{x_i
\}_{i \in \mathbb{N}}$ in $X$ which is said to be $I_\lambda$-lacunary statistically convergent of order $\alpha$ to
$Z$ if for each neighborhood $V$ of 0 and each $\xi>0$,

$$\left\{i \in \mathbb{N}:\frac{1}{\lambda_i^\alpha} \left\vert \left\{ j \in I_i:\frac{1}{h_i} \displaystyle\sum_j
M_j \left(\frac{\vert  x_j-Z \vert}{\rho^{(j)}} \right) \notin V \right\}\right\vert \geq \xi \right\} \in I.$$

In this case, we write $x_j \rightarrow L(S_{I_ \lambda}^\alpha(\mathscr{M},\theta))$ and denote by
$S_{I_\lambda}^\alpha(X,\mathscr{M},\theta)$ the set of all $I_\lambda$-lacunary statistically convergent sequence of
order $\alpha$ in $X$.

\begin{theorem} If $\lambda =(\lambda_i), \mu=(\mu_i) \in \Lambda$ such that $\lambda_i \leq \mu_i$, for all $i \in
\mathbb{N}$ with $\lim \displaystyle\inf_{i \rightarrow \infty} \frac{\lambda_i^\alpha}{\mu_i^\beta}>0$ and $0<\alpha
\leq \beta \leq 1$ for fixed reals $\alpha$ and $\beta$ then $S_{I_\mu}^\beta(X,\mathscr{M}, \theta) \subset
S_{I_\lambda}^\alpha(X,\mathscr{M}, \theta)$. \end{theorem}

\begin{proof} Let us take any neighborhood $V$ of 0.\\ Then, \begin{eqnarray*} && \frac{1}{\mu_i^\beta}
\left\vert\left\{ j \in I_i: \frac{1}{h_i} \displaystyle\sum_j M_j \left( \frac{\vert x_j-Z\vert}{\rho^{(j)}} \right)
\notin V \right\} \right\vert \\ && \geq \frac{1}{\mu_i^\beta} \left\vert \left\{ j \in I_i: \frac{1}{h_i}
\displaystyle\sum_j M_j \left( \frac{\vert x_j-Z\vert}{\rho^{(j)}} \right) \notin V \right\} \right\vert \\ && \geq
\frac{\lambda_i^\alpha}{\mu_i^\beta} \frac{1}{\lambda_i^\alpha} \left\vert\left\{ j \in I_i: \frac{1}{h_i}
\displaystyle\sum_j M_j \left( \frac{\vert x_j-Z\vert}{\rho^{(j)}} \right) \notin V \right\}
\right\vert.\end{eqnarray*}

If $\displaystyle\lim_{i \rightarrow \infty} inf \frac{\lambda_i^\alpha}{\mu_i^\beta} =b$, then the set $\left\{i \in
\mathbb{N}: \frac{\lambda_i^\alpha}{\mu_i^\beta} < \frac{b}{2} \right\}$ is finite. Thus, for $\xi>0$ and any
neighborhood $V$ of 0, \begin{eqnarray*} && \left\{i \in \mathbb{N}: \frac{1}{\lambda_i^\alpha} \left\vert \left \{ j
\in I_i: \frac{1}{h_i} \displaystyle\sum_j M_j \left( \frac{\vert x_j-Z\vert}{\rho^{(j)}} \right) \notin V \right\}
\right\vert \geq \xi \right\}.\\ && \subset \left\{i \in \mathbb{N}: \frac{1}{\mu_i^\beta} \left\vert \left \{ j \leq
i: \frac{1}{h_i} \displaystyle\sum_j M_j \left( \frac{\vert  x_j-Z\vert}{\rho^{(j)}} \right) \notin V \right\}
\right\vert \geq \frac{b}{2} \xi \right\} \cup \left\{ i \in \mathbb{N} : \frac{\lambda_i^\alpha}{\mu_i^\beta} <
\frac{b}{2} \right\}.\end{eqnarray*} So, if $x_i  \rightarrow L(S_{I_\mu}^\beta(\mathscr{M},\theta))$, then the set on
the right hand side belongs to $I$. This completes the proof. \end{proof}

\begin{theorem} Let $0<\alpha \leq \beta \leq 1$ for some fixed reals $\alpha$ and $\beta$ and let $\lambda, \mu \in
\Lambda$ be such that $\lambda_i \leq \mu_i$ . If $\displaystyle\lim_i \frac{\mu_i}{\lambda_i^\alpha}=1$. Then
$S_{I_\lambda}^\alpha (X,\mathscr{M}, \theta) \subset S_{I_\mu}^\beta(X, \mathscr{M},\theta)$. \end{theorem}

\begin{proof}  Let $\xi >0$ be given. Since $\displaystyle\lim_i \frac{\mu_i}{\lambda_i^\alpha} =1$, we can choose $ t
\in \mathbb{N}$ such that,

$$\left\vert \frac{\mu_i}{\lambda_i^\beta}-1 \right\vert < \frac{\xi}{2}, \mbox{~for all ~} i \geq t.$$

Let us take any neighborhood $V$ of 0. Now observe that, \begin{eqnarray*} && \frac{1}{\mu_i^\beta} \left\vert
\left\{j \leq i: \frac{1}{h_i} \displaystyle\sum_j M_j \left( \frac{\vert x_j-Z\vert}{\rho^{(j)}} \right) \notin V
\right\} \right\vert \\ && =\frac{1}{\mu_i^\beta} \left\vert \left\{i-\mu_i+1 < j < i-\lambda_i : \frac{1}{h_i}
\displaystyle\sum_j M_j \left( \frac{\vert  x_j-Z\vert}{\rho^{(j)}} \right) \notin V \right\} \right\vert +
\frac{1}{\mu_i^\beta} \left\vert \left\{j \in I_i: \frac{1}{h_i} \displaystyle\sum_j M_j \left( \frac{\vert
x_j-Z\vert}{\rho^{(j)}} \right) \notin V \right\} \right\vert \\ && <\frac{\mu_i-\lambda_i}{\mu_i^\beta}
+\frac{1}{\mu_i^\beta} \left\vert \left\{j \in I_i: \frac{1}{h_i} \displaystyle\sum_j M_j \left( \frac{\vert
x_j-Z\vert}{\rho^{(j)}} \right) \notin V \right\} \right\vert\\ && <\frac{\mu_i-\lambda_i^\beta}{\lambda_i^\beta}
+\frac{1}{\mu_i^\beta} \left\vert \left\{j \in I_i: \frac{1}{h_i} \displaystyle\sum_j M_j \left( \frac{\vert
x_j-Z\vert}{\rho^{(j)}} \right) \notin V \right\} \right\vert\\ && <\left(\frac{\mu_i}{\lambda_i^\beta}-1 \right)
+\frac{1}{\lambda_i^\alpha} \left\vert \left\{j \in I_i: \frac{1}{h_i} \displaystyle\sum_j M_j \left( \frac{\vert
x_j-Z\vert}{\rho^{(j)}} \right) \notin V \right\} \right\vert\\ && < \frac{\xi}{2}+ \frac{1}{\lambda_i^\alpha}
\left\vert \left\{j \in I_i: \frac{1}{h_i} \displaystyle\sum_j M_j \left( \frac{\vert x_j-Z\vert}{\rho^{(j)}} \right)
\notin V \right\} \right\vert, \mbox{~for all~} i \geq t. \end{eqnarray*} Hence, for $\xi>0$ and any neighborhood $V$
of 0, \begin{eqnarray*} && \left\{ i \in \mathbb{N}: \frac{1}{\mu_i^\beta} \left\vert \left\{j \leq i: \frac{1}{h_i}
\displaystyle\sum_j M_j \left( \frac{\vert x_j-Z\vert}{\rho^{(j)}} \right) \notin V \right\} \right\vert \geq \xi
\right\}\\ && \subset \left\{i \in \mathbb{N}: \frac{1}{\lambda_i^\alpha} \left\vert \left\{j \in I_i: \frac{1}{h_i}
\displaystyle\sum_j M_j \left( \frac{\vert x_j-Z\vert}{\rho^{(j)}} \right) \notin V \right\} \right\vert \geq
\frac{\xi}{2} \right\} \cup \{1,2,...t \}.\end{eqnarray*}

If $x_i \rightarrow L(S_{I_\lambda}^\alpha(\mathscr{M},\theta))$,then the set on the right hand side belongs to $I$
and so the set  on the left hand side also belongs to $I$. This shows that $(x_i)$ is $I_\mu$-lacunary statistically
convergent of order $\beta$ to $Z$. \end{proof}


\section*{Competing interests } The authors declare that they have no competing interests.


\section*{Acknowledgement} The authors gratefully acknowledge that the present work partially supported by the University Putra Malaysia(UPM) during second authors visit UPM. 



\end{document}